\documentclass{svmult}
\usepackage{makeidx}         
\usepackage{graphicx}        
\usepackage{multicol}        
\usepackage[bottom]{footmisc}
\usepackage{natbib}
\usepackage{url}

\renewcommand{\PackageWarningNoLine}[2]{}
\usepackage{amsmath}
\usepackage{amsfonts}
\usepackage{color}

\begin{document}
\title*{Discontinuous Galerkin Isogeometric Analysis of Elliptic PDEs on Surfaces}
\author{Ulrich Langer\inst{1}\and
Stephen E. Moore\inst{2}}
\institute{Institute for Computational Mathematics, Johannes Kepler University,       	
Altenbergerstr. 69, A-4040 Linz, Austria  
\texttt{ ulanger@numa.uni-linz.ac.at}
\and Johann Radon Institute for Computational and Applied Mathematics, Austrian Academy of Sciences,
Altenbergerstr. 69, A-4040 Linz, Austria
\texttt{stephen.moore@ricam.oeaw.ac.at}}
\maketitle

\section{Introduction}
\label{sec1:Introduction}

The Isogeometric Analysis (IGA) 
was introduced by \cite{HughesCottrellBazilevs:2005a}
and has since been developed intensively, see also monograph 
\cite{CottrellHughesBazilevs:2009a}, is a very suitable framework 
for representing and discretizing
Partial Differential Equations (PDEs) on surfaces.
We refer the reader to the survey paper by 
\cite{DziukElliot:2013a} 
where different finite element approaches to the numerical solution 
of PDEs on surfaces are discussed. 
Very recently, 
\cite{DednerMadhavanStinner:2013a}
have used and analyzed the Discontinuous Galerkin (DG) finite element method 
for solving elliptic problems on surfaces. 
The IGA of second-order PDEs on surfaces 
that avoid errors arising from the approximation of the surface,  
has been introduced and numerically studied by 
\cite{DedeQuarteroni:2012a}.
\cite{Brunero:2012a} presented some
discretization error analysis 
of the DG-IGA applied to plane (2d) diffusion problems that carries over to 
plane linear elasticity problems which have recently been studied 
numerically in 
\cite{ApostolatosSchmidtWuencherBletzinger:2013a}.
The efficient generation of the IGA equations, their fast solution,
and the implementation of adaptive IGA schemes are currently 
hot research topics. The use of DG technologies will certainly 
facilitate the handling of the multi-patch case. 
\par In this paper, we use the DG method to handle 
the IGA of diffusion problems on closed or open, multi-patch NURBS surfaces. 
The DG technology easily allows us to handle non-homogeneous Dirichlet 
boundary condition as in the Nitsche method and the multi-patch 
NURBS spaces which can be discontinuous across the patch boundaries.
We also derive discretization error estimates 
in the DG- and $L_{2}$-norms. Finally, we present some numerical results 
confirming our theoretical estimates.
%
\section{Surface Diffusion Model Problem}
\label{sec2:SurfaceDiffusionModelProblem}

Let us assume that the physical (computational) domain $\Omega$,
where we are going to solve our diffusion problem,
is a sufficiently smooth, two-dimensional generic (Riemannian) manifold (surface)
defined in the physical space $\mathbb{R}^{3}$ 
by means of a smooth multi-patch NURBS mapping
that is defined as follows. 
Let $\mathcal{T}_{H}= \{\Omega^{(i)}\}_{i=1}^{N}$ be a partition of 
our physical computational domain $\Omega$ into 
non-overlapping patches (subdomains) $\Omega^{(i)}$ such that 
$\overline{\Omega}= \bigcup_{i=1}^{N} \overline{\Omega}^{(i)} $ and 
$\Omega^{(i)} \cap \Omega^{(j)}= \emptyset $ for $i \neq j$,
and let each patch $\Omega^{(i)}$ be the image of the 
parameter domain  $\widehat{\Omega} = (0,1)^2 \subset \mathbb{R}^{2}$ 
by some NURBS mapping 
$G^{(i)} : \widehat{\Omega} \rightarrow \Omega^{(i)} \subset \mathbb{R}^{3}, 
\mathbf{\xi} = (\mathbf{\xi}_1,\mathbf{\xi}_2) 
\mapsto \mathbf{x} = (\mathbf{x}_1,\mathbf{x}_2,\mathbf{x}_3)=G^{(i)}(\mathbf{\xi})$, 
which can be represented in the form
\begin{equation}
\label{sec2:GeometricalMappingRepresentation}
  G^{(i)}(\xi_{1},\xi_{2}) = \sum_{k_{1}=1}^{n_{1}} \sum_{k_{2}=1}^{n_{2}}
 \mathbf{P}^{(i)}_{(k_{1},k_{2})} \widehat{R}^{(i)}_{(k_{1},k_{2})}(\xi_{1},\xi_{2})
\end{equation}
where $\{ \widehat{R}^{(i)}_{(k_{1},k_{2})} \}$ are the bivariate NURBS basis functions, 
and  $\{\mathbf{P}^{(i)}_{(k_{1},k_{2})} \}$ are the control points, 
see \cite{CottrellHughesBazilevs:2009a} for a detailed 
description.

\par Let us now consider a 
diffusion problem 
on the surface $\Omega$ the weak formulation of which 
can be written as follows: find $u \in V_g$ such that
\begin{equation}
\label{sec2:VariationalFormulation}
	a(u,v) = \langle  F,v \rangle \quad \forall v \in V_0,
\end{equation}
with the bilinear and linear forms are given by the relations
\begin{equation*}
	a(u,v) = \int_\Omega \alpha \, \nabla_\Omega u \cdot \nabla_\Omega v \, d \Omega
	 \quad \mbox{and} \quad
	\langle  F,v \rangle  = \int_\Omega f v \, d \Omega + \int_{\Gamma_N} g_N v \,d \Gamma,
\end{equation*}
respectively, where $\nabla_\Omega$ denotes the so-called tangential or surface
gradient, see e.g. Definition~2.3 in \cite{DziukElliot:2013a}
for its precise description.
The hyperplane $V_g$ and the test space $V_0$ are given by 
	$V_g=\{v \in V = H^1(\Omega): v=g_D \;\mbox{on}\; \Gamma_D\}$
and
	$V_0=\{v \in V: v=0 \;\mbox{on}\; \Gamma_D\}$
for the case of an open surface $\Omega$ with the boundary 
$\Gamma = \overline{\Gamma}_D \cup \overline{\Gamma}_N$ 
such that $\mbox{meas}_1(\Gamma_D) > 0$, 
whereas
	$V_g=V_0=\{v \in V: \int_\Omega v \, d \Omega =0\}$
in the case of a pure Neumann problem ($\Gamma_N = \Gamma$)
as well as in the case of closed surfaces unless there is 
a reaction term. In case of closed
surfaces there is of course  no integral over $\Gamma_N$ 
in the linear functional on the right-hand side of
(\ref{sec2:VariationalFormulation}).
In the remainder of the paper, we will mainly discuss the case of mixed boundary 
value problems on an open surface under appropriate  assumptions 
(e.g., $\mbox{meas}_1(\Gamma_D) > 0$, $\alpha$ - uniformly positive and bounded,
$f\in L_2(\Omega)$, $g_D \in H^{\frac{1}{2}}(\Gamma_{D})$ and $g_{N} \in L_{2}(\Gamma_{N})$~) 
ensuring existence and uniqueness 
of the solution of (\ref{sec2:VariationalFormulation}).
For simplicity, we assume that the diffusion coefficient $\alpha$ is patch-wise constant,
i.e. $\alpha = \alpha_i$ on $\Omega^{(i)}$ for $i=1,2,\ldots,N$.
The other cases including the reaction-diffusion case can be treated in the same way 
and yield the same results like presented below.

\section{DG-IGA Schemes and their Properties}
\label{sec3:DGIGASchemesAndProperties}
The DG-IGA variational identity
\begin{equation}
 \label{sec3:DG-VariationalIdentity}
a_{DG}(u,v) = \langle  F_{DG},v \rangle \quad \forall v \in \mathcal{V} = H^{1+s}(\mathcal{T}_{H}),
\end{equation}
which corresponds to (\ref{sec2:VariationalFormulation}),
can be derived in the same way as their FE counterpart, where 
$H^{1+s}(\mathcal{T}_{H}) =\{v \in L_{2}(\Omega): v|_{\Omega^{(i)}} 
\in H^{1+s}(\Omega^{(i)}), \; \forall \, i = 1,\ldots,N\}$
with some $s > 1/2$. 
The DG bilinear and linear forms in the 
\textit{Symmetric Interior Penalty Galerkin} (SIPG) version, 
that is considered throughout this paper for definiteness,
are defined by the relationships
\begin{eqnarray}
\label{sec3:DG-BilinearForm}
\nonumber
 a_{DG}(u,v) &=&\sum_{i=1}^{N} \int_{\Omega^{(i)}} 
		  \alpha_{i} \nabla_{\Omega} u \cdot \nabla_{\Omega} v \, d\Omega\\ \nonumber
	     && -\sum_{\gamma \in \mathcal{E}_{I} \cup \mathcal{E}_{D}} 
		  \int_{\gamma} 
		  \left( 
		  \{ \alpha \nabla_{\Omega} u \cdot \mathbf{n}\} 
		  [v] +
		  \{\alpha \nabla_{\Omega} v \cdot \mathbf{n}\}  [u]
		  \right)\,d\Gamma\\
	     && + \sum_{ \gamma \in \mathcal{E}_{I} \cup \mathcal{E}_{D}} 
		  \frac{\delta}{ h_{\gamma} }
		  \int_{\gamma}  \alpha_{\gamma} [u][v]\,d\Gamma
\end{eqnarray}
and
\begin{eqnarray}
\label{sec3:DG-LinearForm}
\nonumber
\langle  F_{DG},v \rangle  
  &=& \int_{\Omega} f v d\,\Omega 
  + \sum_{\gamma \in \mathcal{E}_{N}} \int_{\gamma} g_{N}v\, d\Gamma\\
  && + \sum_{\gamma \in \mathcal{E}_{D}} \int_{\gamma} 
		\alpha_{\gamma}	  \left( -  \nabla_{\Omega} v \cdot \mathbf{n}  
			  +  \frac{\delta}{ h_{\gamma} } v \right) g_{D}\,d\Gamma,
\end{eqnarray}
respectively, where the usual DG notations 
for the averages $\{ \cdot \}$ and jumps $[\cdot ]$
are used,
see, e.g., \cite{Riviere:2008a}.
The sets 
$\mathcal{E}_{I}$, $\mathcal{E}_{D}$ and $\mathcal{E}_{N}$ 
denote the sets of edges $\gamma$ of the patches 
belonging to $\Gamma_I = \cup \,\partial \Omega^{(i)} \setminus \{\Gamma_D \cup \Gamma_N\}$,
$\Gamma_D$ and $\Gamma_N$, respectively whereas 
$h_{\gamma}$ is the mesh-size on $\gamma$.
The penalty parameter $\delta$ must be chosen such that 
the ellipticity of the DG bilinear on $\mathcal{V}_{h}$
can be ensured.
The relationship between our model problem 
(\ref{sec2:VariationalFormulation}) 
and the DG variational identity
(\ref{sec3:DG-VariationalIdentity})
is given by the consistency theorem that can easily be verified.
\begin{theorem}\label{Thm:Sec3:Consistency}
If the solution $u$ of the variational problem (\ref{sec2:VariationalFormulation}) belongs 
to $V_g \cap H^{1+s}(\mathcal{T}_{H  })$ with some $s > 1/2$, then
$u$ satisfies the DG variational identity (\ref{sec3:DG-VariationalIdentity}). 
Conversely, if $u \in  H^{1+s}(\mathcal{T}_{H})$
satisfies (\ref{sec3:DG-VariationalIdentity}), 
then $u$ is the solution of our original variational problem 
(\ref{sec2:VariationalFormulation}).
\end{theorem}

\par Now we consider the finite-dimensional Multi-Patch NURBS subspace
\begin{equation*}
\mathcal{V}_{h}= \{v \in L_{2}(\Omega): \; v|_{\Omega^{(i)}}\in V^{i}_{h}(\Omega^{(i)}), \; i= 1,\ldots,N \}
\end{equation*}
of our DG space $\mathcal{V}$,
where 
$
V^{i}_{h}(\Omega^{(i)}) = \text{span}\{R_{\textbf{k}}^{(i)} \}
$
denotes the space of NURBS functions on each single-patch $ \Omega^{(i)}, \; i= 1,\ldots,N$, 
and the NURBS basis functions
$
  R_{\textbf{k}}^{(i)} = \widehat{R}^{(i)}_{\textbf{k}} \circ G^{(i)^{-1}} 
$
are given by the push-forward of the NURBS functions $\widehat{R}^{(i)}_{\textbf{k}}$ 
to their corresponding physical sub-domains $ \Omega^{(i)}$ on the surface $\Omega$. 
Finally, the DG scheme for our model problem \eqref{sec2:VariationalFormulation} 
reads as follows:
find $u_{h} \in \mathcal{V}_{h}$ such that 
\begin{equation}
\label{sec3:DiscreteDGVariationalFormulation}
 a_{DG}(u_{h},v_{h}) = \langle F_{DG},v_{h} \rangle , \quad \forall v_{h} \in \mathcal{V}_{h}.
\end{equation}

For simplicity of our analysis, we assume matching meshes in the IGA sense, 
where the discretization parameter $h_i$ characterizes the mesh-size in the patch $\Omega^{(i)}$
whereas $p$ always denotes the underlying polynomial degree of the NURBS. 
Using special trace and inverse inequalities in the NURBS spaces $\mathcal{V}_{h}$ 
and Young's inequality, for sufficiently large DG penalty parameter $\delta$,
we can easily establish $\mathcal{V}_{h}$ coercivity and boundedness 
of the DG bilinear form with respect to the DG energy norm 
\begin{equation}
\label{sec3:DG-Norm}
\|v \|^{2}_{DG} = 
	\sum_{i=1}^{N} \alpha_{i} \|\nabla_{\Omega} v_{i} \|_{L^{2}(\Omega^{i})}^{2}   
        + \sum_{ \gamma \in \mathcal{E}_{I} \cup \mathcal{E}_{D}}
	 \alpha_{\gamma} \frac{\delta}{h_{\gamma}} \| [v] \|_{L^{2}(\gamma)}^{2}, 
\end{equation}
yielding existence and uniqueness of the DG solution $u_{h} \in \mathcal{V}_{h}$
of (\ref{sec3:DiscreteDGVariationalFormulation})
that can be determined  by the solution of a linear system of algebraic
equations.

\section{Discretization Error Estimates}
\label{sec4:DiscretizationErrorEstimates}
\begin{theorem}\label{sec4:DG-Norm-ErrorEstimate}
Let $u \in V_g \cap H^{1+s}(\mathcal{T}_{H})$ with some $s > 1/2$ be the solution of 
(\ref{sec2:VariationalFormulation}),
$u_{h} \in \mathcal{V}_{h}$ be the solution of (\ref{sec3:DiscreteDGVariationalFormulation}), 
and the penalty parameter $\delta$  be chosen large enough . 
Then there  exists a positive constant $c$
that is independent of $u$, the discretization parameters
and the jumps in the diffusion coefficients
such that the DG-norm error estimate 
\begin{equation}
\label{sec4:DG-normError Estimate}
      \|u-u_{h} \|_{DG}^{2} \leq c 
	\left(\sum_{i=1}^{N} \alpha_{i} h_{i}^{2t}\|u\|^{2}_{H^{1+t}(\Omega^{(i)})}\right)^{1/2},
\end{equation}
holds with $t:= \min\{s,p\} $.
\end{theorem}
\begin{proof}
Let us give a sketch of the proof.
By the triangle inequality, we have 
\begin{equation}\label{sec4:DGTriangleInequality}
  \|u-u_{h} \|_{DG} \leq \|u- \Pi_{h} u\|_{DG} + \|\Pi_{h} u - u_{h}\|_{DG}
\end{equation}
with some quasi-interpolation operator $\Pi_{h}: \mathcal{V} \mapsto \mathcal{V}_{h}$
such that the first term can be estimated with optimal order, i.e. by the term 
on the right-hand side of (\ref{sec4:DG-normError Estimate})
with some other constant $c$. This is possible due to the 
approximation results known for NURBS, 
see, e.g., \cite{BazilevsBeiraoCottrellHughesSangalli:2006a} 
and \cite{CottrellHughesBazilevs:2009a}.
Now it remains to estimate the second term in the same way.
Using the Galerkin orthogonality $a_{DG}(u-u_h,v_h)=0$ for all $v_h \in \mathcal{V}_{h}$,
the $\mathcal{V}_{h}$ coercitivity of the bilinear form $a_{DG}(\cdot,\cdot)$, 
the scaled trace inequality
\begin{equation}
\label{sec4:EpsilonTraceInequality}
  \|v \|_{L^{2}(e)} \leq 
	C h^{-1/2}_{E} \left( \|v \|_{L^{2}(E)} 
	      + h^{1/2+\epsilon}_{E} |v|_{H^{1/2+\epsilon}(E)} \right),
\end{equation}
that holds for all $v \in H^{1/2+\epsilon}(E)$, for all IGA mesh elements $E$, 
for all edges $e \subset \partial E$, and for  $\epsilon > 0$,
where $ h_{E}$ denotes the mesh-size of $E$  or the length of $e$, 
Young's inequality, and again the approximation properties 
of the quasi-interpolation operator $\Pi_{h}$, we can estimate the second 
term by the same term
$ c \left(\sum_{i=1}^{N} \alpha_{i} h_{i}^{2t}\|u\|^{2}_{H^{1+t}(\Omega^{(i)})}\right)^{1/2}$
with some (other) constant $c$.
This completes the proof of the theorem.
\end{proof}

Using duality arguments, we can also derive $L_2$-norm 
error estimates that depend on the elliptic regularity.
Under the assumption of full elliptic regularity, we get 
$\|u-u_{h} \|_{L_2(\Omega)} \le c\, h^{p+1} \|u\|_{H^{p+1}(\Omega)}$
that is nicely confirmed by our numerical experiments
presented in the next section for $p=1,2,3,4$.

\section{Numerical Results}
\label{sec5:NumericalResults}
The DG IGA method presented in this paper 
as well as its continuous 
Galerkin
counterpart have been implemented
in the  object oriented C++ IGA library ''Geometry + Simulation Modules''
(G+SMO)~\footnote{G+SMO : https://ricamsvn.ricam.oeaw.ac.at/trac/gismo}.  
We present some first numerical results for testing 
the numerical behavior of
the discretization error with respect to the mesh parameter $h$
and the polynomial degree $p$ Concerning the choice of the penalty parameter, 
we used $\delta = 2(p+2)(p+1).$

\par As a first example, we consider a non-homogeneous Dirichlet problem for 
the Poisson equation 
in the 2d computational domain 
$\Omega \subset \mathbb{R}^2$ 
called Yeti's footprint, see also \cite{KleissPechsteinJuttlerTomar:2012a},
where the right-hand side $f$ and the Dirichlet data $g_D$ are chosen such 
that $u(x_1,x_2) =  \sin(\pi x_1)\sin(\pi x_2)$ is the solution of the boundary value problem.
The computational domain (left) and the solution (right) can be seen in 
Fig.~\ref{sec5:fig1:YetiFoot}.
The Yeti footprint consists of 21 patches with varying open knot vectors $\Xi$
describing the 
NURBS
discretization in a short and precise way,
see, e.g., \cite{CottrellHughesBazilevs:2009a}
for a detailed definition. 
The knot vector for patches 1 to 16 and 21 is given as
$\Xi = (0,\ldots,0,0.5,1,\ldots,1)$ in both directions whereas
the knot vectors for the patches 17 to 20 are given as 
$\Xi_1 = ( 0,\ldots,0,0.5,1,\ldots,1)$ and 
$\Xi_2 = ( 0,\ldots,0,0.25,0.5,0.75,1,\ldots,1).$
\begin{figure}
  \centering
  \includegraphics[width  = 0.35\textwidth]{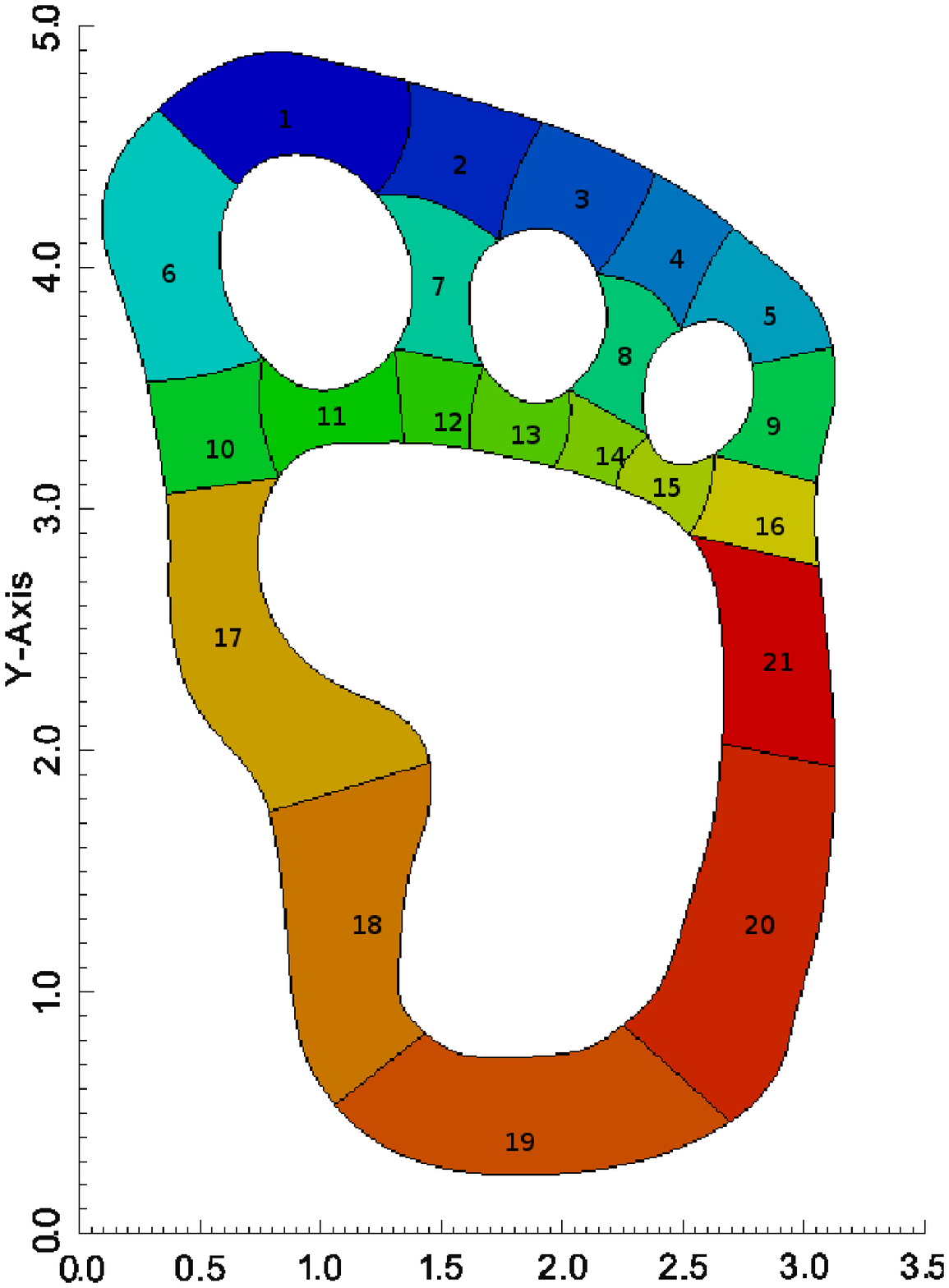} \hfil
  \includegraphics[width = 0.35\textwidth]{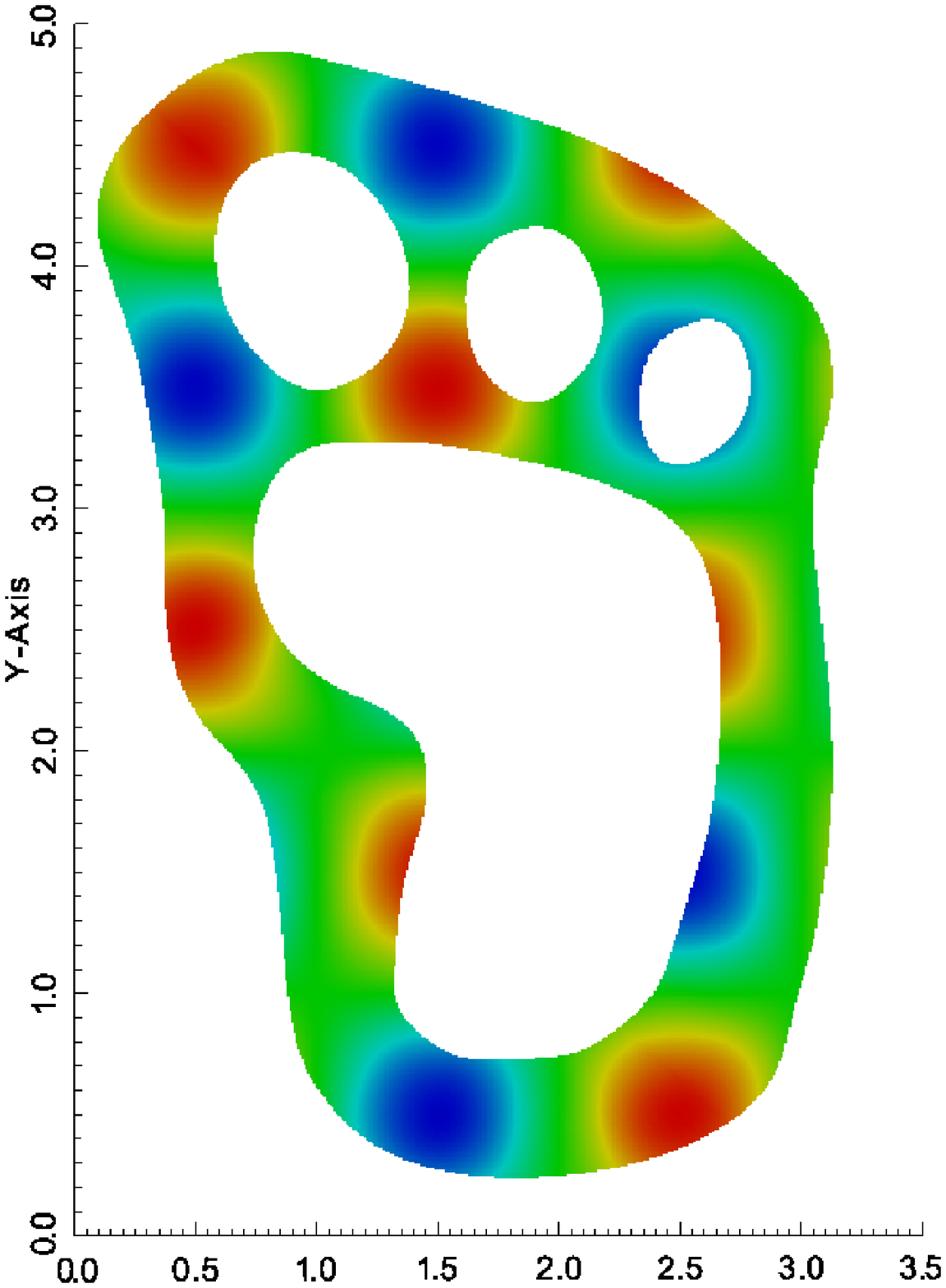}
  \caption{Yeti foot: geometry (left) and DG-IGA solution (right)}
  \label{sec5:fig1:YetiFoot}  
\end{figure}
\begin{figure}
  \centering
  \includegraphics[width  = 0.7\textwidth]{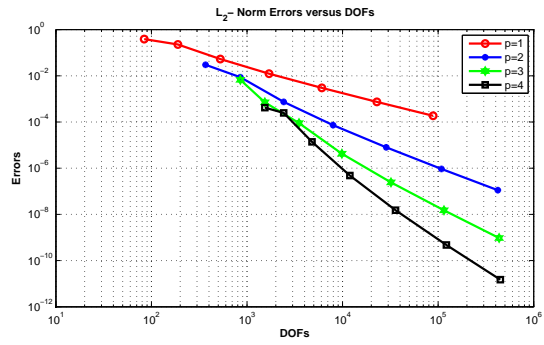} 
  \caption{Yeti foot: $L_2-$ Norm Errors  with polynomial degree $p$}
  \label{sec5:fig2:YetiFootL2Error}  
\end{figure}
\begin{figure}
  \centering
  \includegraphics[width  = 0.7\textwidth]{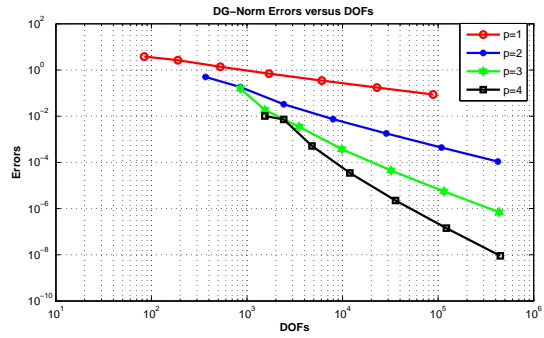} 
  \caption{Yeti foot: DG Norm  Errors with polynomial degree $p$}
  \label{sec5:fig3:YetiFootDGError}     
\end{figure}
In Fig.~\ref{sec5:fig2:YetiFootL2Error} and \ref{sec5:fig3:YetiFootDGError}, 
the errors in the $L_2$-norm and in the DG energy norm (\ref{sec3:DG-Norm})
are plotted against the degree of freedom (DOFs) with polynomial degrees from 1 to 4.  
It can be observed that we have convergence rates of $\mathcal{O}(h^{p+1})$ 
and $\mathcal{O}(h^{p})$ respectively. 
This corresponds to our theory in Section~\ref{sec4:DiscretizationErrorEstimates}.

In the second example, we apply the DG-IGA to the same 
Laplace-Beltrami problem on an open surface 
as described in \cite{DedeQuarteroni:2012a}, section~5.1, 
where $\Omega$ is a quarter cylinder represented by four patches in our computations, 
see Fig.~\ref{sec5:fig4:MultiPatchCylinder} (left).
The $L_{2}-$norm errors plotted on the right side of
Fig.~\ref{sec5:fig4:MultiPatchCylinder}
exhibit the same numerical behavior as in the plane 
case of the Yeti foot.
The same is true for the DG-norm.
\begin{figure}[bth!]
  \centering
  \includegraphics[width  = 0.38\textwidth]{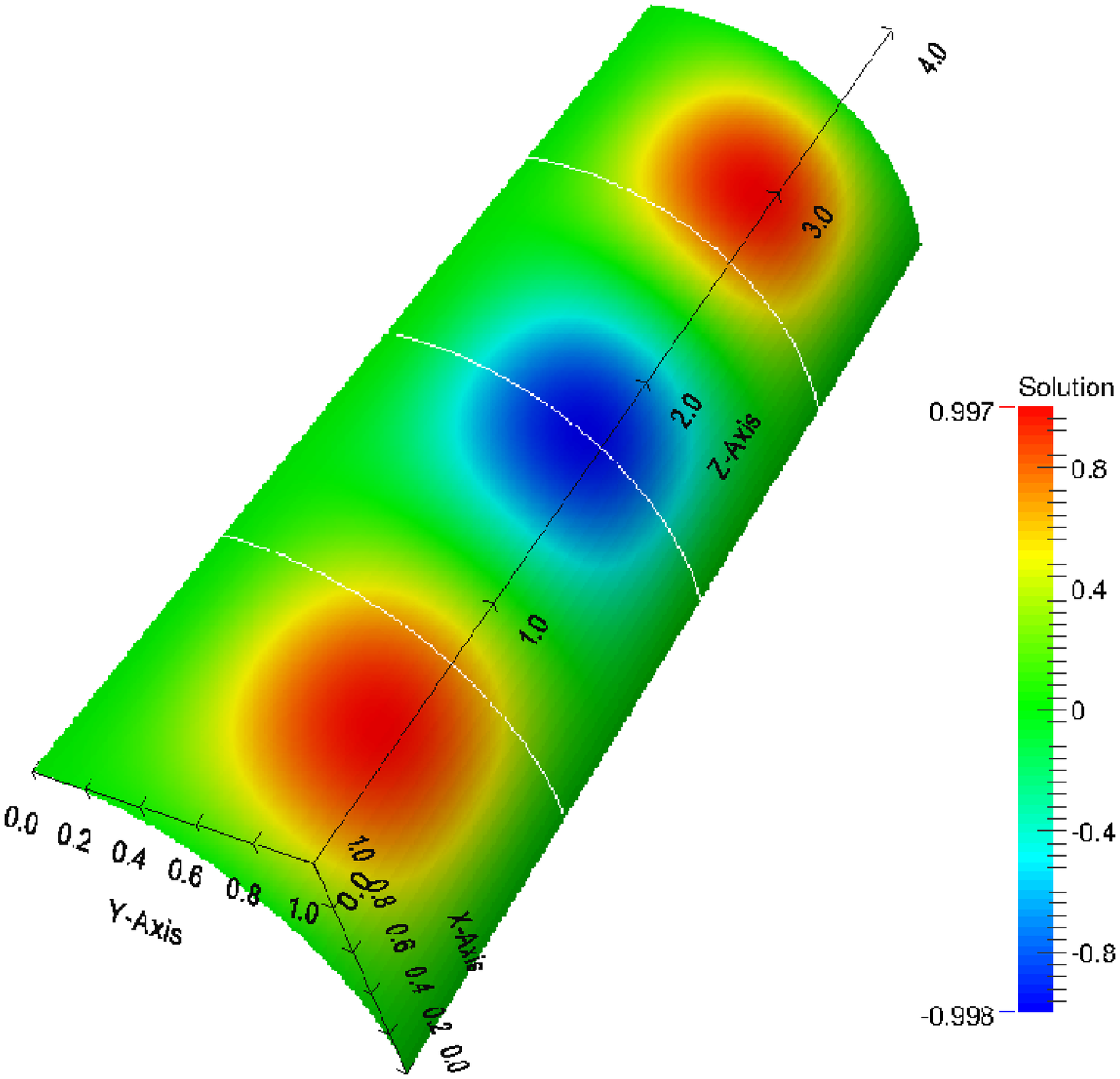} 
  \includegraphics[width  = 0.60\textwidth]{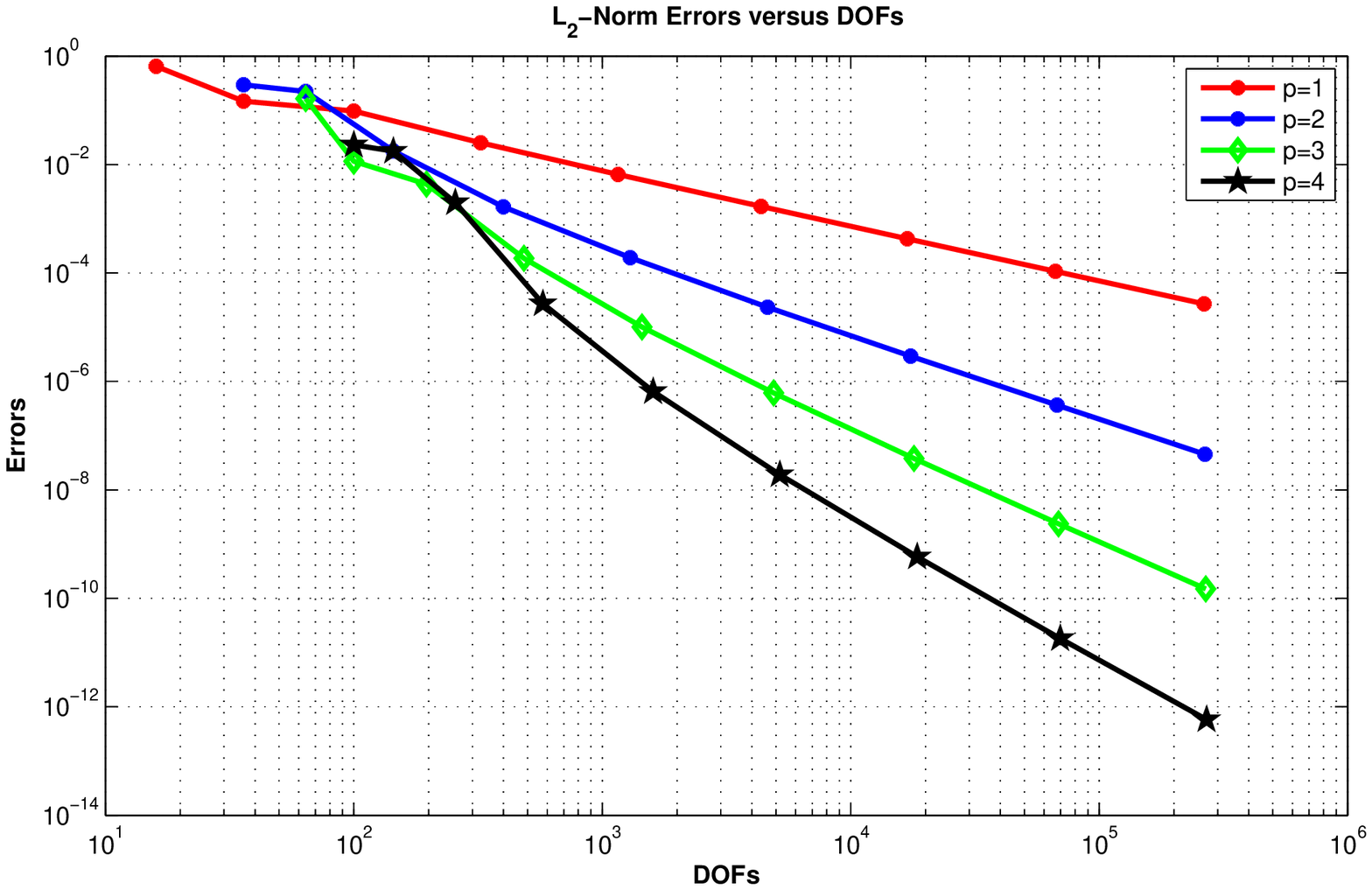} 
  \caption{Quarter cylinder: geometry with the solution  (left) 
      and $L_{2}$ norm errors (right).}
\label{sec5:fig4:MultiPatchCylinder}  
\end{figure}

\section{Conclusions}
\label{sec6:Conclusion}
We have developed and analyzed a new method for the numerical approximation 
of diffusion problems on open and closed surfaces 
by combining the discontinuous Galerkin technique with isogeometric analysis.
We refer to our approach as the Discontinuous Galerkin Isogeometric Analysis (DG-IGA).
In our DG approach we allow discontinuities only across the 
boundaries of the patches, 
into which the computational domain is decomposed, 
and enforce the interface conditions in the DG framework. 
For simplicity of presentation, we assume that 
the meshes are matching across the patches, and
the solution $u$ 
is at least patch-wise in $H^{1+s}$, i.e. $u \in H^{1+s}(\mathcal{T}_{H})$,
with some $s > 1/2$. 
The cases of non-matching meshes and low-regularity solution,  
that are technically more involved and that were investigated,
e.g., by  \cite{Dryja:2003a} and \cite{DiPietroErn:2012a},
will be considered in a forthcoming paper.
The parallel solution of the DG-IGA equations can efficiently be performed by 
Domain Decomposition (DD) solvers like the IETI technique
proposed by \cite{KleissPechsteinJuttlerTomar:2012a},
see also \cite{ApostolatosSchmidtWuencherBletzinger:2013a}
for other DD solvers. The construction and analysis of 
efficient solution strategies is currently a hot research topic
since, beside efficient generation techniques, 
the solvers are the efficiency bottleneck
in large-scale IGA computations.

\section*{Acknowledgement}
The authors gratefully acknowledge the financial support of the
research project NFN S117-03 by the Austrian Science Fund (FWF). 
Furthermore, the authors want to thank their colleagues
Angelos Mantzaflaris, Satyendra Tomar, Ioannis Toulopoulos and Walter Zulehner 
for fruitful and enlighting discussions as well as for their 
help in the implementation in GISMO.


\bibliographystyle{plainnat}
\bibliography{DGIGASurface_MooreLanger} 

\end{document}